\documentclass[12pt]{article}

\usepackage{graphicx}
\usepackage{nameref}
\usepackage[shortlabels,inline]{enumitem}
\usepackage{empheq}
\usepackage[shortlabels]{enumitem}
\setlist[enumerate]{nosep}
\usepackage[doc]{optional}
\usepackage{xcolor}
\usepackage[colorlinks=true,
            linkcolor=refkey,
            urlcolor=lblue,
            citecolor=red]{hyperref}
\usepackage{float}
\usepackage{soul}
\usepackage{graphicx}
\definecolor{labelkey}{rgb}{0,0.08,0.45}
\definecolor{refkey}{rgb}{0,0.6,0.0}
\definecolor{Brown}{rgb}{0.45,0.0,0.05}
\definecolor{lime}{rgb}{0.00,0.8,0.0}
\definecolor{lblue}{rgb}{0.5,0.5,0.99}

 \usepackage{mathpazo}

\colorlet{hlcyan}{cyan!30}

\usepackage{stmaryrd}
\usepackage{amssymb}
\makeatletter
\def\namedlabel#1#2{\begingroup
   \def\@currentlabel{#2}%
   \label{#1}\endgroup
}
\makeatother

\newcommand{\To}{\ensuremath{\rightrightarrows}}

\newcommand{\fenv}[1]%
{\ensuremath{\,\overrightarrow{\operatorname{env}}_{#1}}}
\newcommand{\benv}[1]%
{\ensuremath{\,\overleftarrow{\operatorname{env}}_{#1}}}

\newcommand{\scal}[2]{\left\langle{#1},{#2}  \right\rangle}

\newcommand{\RR}{\ensuremath{\mathbb R}}

\newcommand{\dom}{\ensuremath{\operatorname{dom}}}

\newcommand{\zer}{\ensuremath{\operatorname{zer}}}

\newcommand{\Fix}{\ensuremath{\operatorname{Fix}}}
\newcommand{\Id}{\ensuremath{\operatorname{Id}}}
{\begin{list}{}{%
\settowidth{\labelwidth}{\textrm{#1~}}%
\setlength{\leftmargin}{\labelwidth+\labelsep}}}%requires macro calc.sty
{\end{list}}
\usepackage{amsthm}
\usepackage[capitalize,nameinlink]{cleveref}
\crefname{equation}{}{equations}
\crefname{chapter}{Appendix}{chapters}
\crefname{item}{}{items}
\crefname{enumi}{}{}
\newtheorem{theorem}{Theorem}[section]
\newtheorem{lemma}[theorem]{Lemma}

\newtheorem{proposition}[theorem]{Proposition}

\newtheorem{definition}[theorem]{Definition}

\newtheorem{ex}[theorem]{Example}

\newtheorem{remark}[theorem]{Remark}

\providecommand{\norm}[1]{\lVert#1\rVert}

\providecommand{\RR}{\mathbb{R}}

\providecommand{\dom}{\operatorname{dom}}

\newcommand{\fix}{\ensuremath{\operatorname{Fix}}}

\providecommand{\gra}{\operatorname{gra}}
\providecommand{\Id}{\operatorname{{ Id}}}

\providecommand{\To}{\rightrightarrows}

\providecommand{\fix}{\operatorname{Fix}}

\providecommand{\Id}{\operatorname{Id}}

\providecommand{\zer}{\operatorname{zer}}

\providecommand{\RR}{\mathbb{R}}

\definecolor{myblue}{rgb}{.8, .8, 1}
  \newcommand*\mybluebox[1]{%
    \colorbox{myblue}{\hspace{1em}#1\hspace{1em}}}

\allowdisplaybreaks % or locally if problems {\allowdisplaybreaks
\begin{document}
%-------------------------------------------------------------------------

%\tikzstyle{decision} = [diamond, draw, fill=blue!50]
%\tikzstyle{line} = [draw, -stealth, thick]
%\tikzstyle{elli}=[draw, ellipse, fill=red!50,minimum height=8mm, text width=5em, text centered]
%\tikzstyle{block} = [draw, rectangle, fill=blue!50, text width=8em, text centered, minimum height=15mm, node distance=10em]
%

\title{\textsc{Compositions of Resolvents: Fixed Points Sets and Set of Cycles }}
\author{
Salihah Thabet Alwadani\thanks{
Mathematics, Royal Commission Yanbu Colleges And Institutes, 46455, Yanbu,
Saudi Arabia. E-mail:
\texttt{salihah.s.alwadani@gmail.com}.}
}

\date{June 3, 2024 }
\maketitle

\vskip 8mm

\begin{abstract} 
In this paper, we investigate the cycles and fixed point sets of compositions of resolvents using Attouch--Th{\'e}ra duality. We demonstrate that the cycles defined by the resolvent operators can be formulated in Hilbert space as the solution to a fixed point equation. Furthermore, we introduce the relationship between these cycles and the fixed point sets of the composition of resolvents.
\end{abstract}

{\small
\noindent
{\bfseries 2020 Mathematics Subject Classification:}
% https://zbmath.org/static/msc2020.pdf
{Primary 47H09, 47H05; Secondary 47A06, 90C25
}

}

\maketitle

\section{Introduction}
Throughout, we assume that 
\begin{empheq}[box=\mybluebox]{equation*}
\text{$X$ is
a real Hilbert space with inner product
$\scal{\cdot}{\cdot}\colon X\times X\to\RR$, }
 \end{empheq}    
and induced norm $\|\cdot\|\colon X\to\RR\colon x\mapsto \sqrt{\scal{x}{x}}$. For more details about Hilbert space we refere the redear to \cite{berberian1961introduction} and \cite{BC2017}. An operator $T: X \to X$ is \emph{nonexpansive} if it is Lipschitz continuous with constant $1$, i.e.,
 \begin{equation}\label{footnote11}
 \big(  \forall x \in X\big) \big( \forall y \in X \big) \ \ \ \norm{Tx - Ty} \leq \norm{x-y}.
 \end{equation}
Nonexpansive operators have a major role in optimization because of the set of fixed points $\fix R := \{  x \in X \mid x = Rx \}$ usually represents solutions to inclusion problems and optimization. For more details about the nonexpansive operators and the fixed point set we refer the reader to \cite{BC2017, bauschke1996approximation, alwadani2020fixed, alwadani2020projection, alwadani2020proximal, bauschke1993convergence, bauschke1994dykstra, bauschke1996approximation, bauschke1997method, BCL2004, cegielski2012iterative, cheney1959proximity}
and \cite[Chapter 3 and Chapter 6]{alwadani2021thesis}.
Moreover, $T: D \to X$ is \emph{firmly nonexpansive} if 
\begin{equation}
\big( \forall x \in D\big) \big( \forall y \in D\big) \ \ \norm{Tx - Ty}^2+ \norm{( \Id - T) x - ( \Id - T) y}^2 \leq \norm{x - y}^2.
\end{equation}
Firmly nonexpansive operators are also central because of nice convergence properties for iterates and their correspndence with maximal monotone operators. Recall that a set-valued operator 
 $A: X \To X$ with graph $\gra A$ is \emph{monotone} if 
 \begin{equation*}
 \big(  \forall \big(   x, u\big) \in \gra A \big) \big(  \forall \big(   y, v\big) \in \gra A \big) \ \ \ \ \  \scal{x-y}{u-v} \geq 0.
 \end{equation*}
Furthermore, $A$ is \emph{maximally monotone} if there exists no monotone operator $B: X \To X$ such that $\gra B$ properly contains $\gra A$, i.e., for every $(x, u) \in X \times X$, 
 \begin{equation*}
 \big(  x, u\big) \in \gra A \ \ \Leftrightarrow \ \ \big( \forall \big( y,v \big) \in \gra A \big) \ \ \scal{x-y}{u-v} \geq 0.
 \end{equation*}
 It is well known that monotone and maximally monotone operators have central roles in various areas of modern nonlinear analysis, see \cite{BC2017}, \cite{borwein2010convex}, \cite{burachik2008enlargements}, \cite{MM}, \cite{MMN}, \cite{zeidler2013nonlinear1}, \cite{zeidler2013nonlinear2}, \cite{zeidler2013nonlinear3}, and \cite{martinez2005monotone} for background meterial. 
 Let $A:X \To X$ be maximally monotone operator and denote the associated  \emph{resolvent} by
\begin{equation}\label{quli1}
J_{A} := (\Id + A)^{-1}.
\end{equation}
In \cite{minty1962monotone}, Minty observed that $J_{A}$ is a firmly nonexpansive operators from $X$ to $X$. For more information about the relation between firmly nonexpansive mappings and maximally monotone operators, see \cite{bauschke2012firmly}. 
The Hilber product space, 
 \begin{empheq}[box=\mybluebox]{equation*}
 \text{$\mathbf{X} =  \Big\{  \mathbf{x} = (x_i)_{i \in I} \mathrel{\Big|} ( \forall i \in I) \ x_i \in X \Big\}$, }
 \end{empheq}
 where $m \in \{ 2, 3, \dots  \}$ and $ i = \{  1, 2, \dots, m\}$. Let 
\begin{empheq}[box = \mybluebox]{equation}
\text{$A_{i}: X \To X$ be maximally monotone operators,} 
\end{empheq}
with resolvents $\textrm{J}_{A_1}$, $\textrm{J}_{A_2}, \dots, \textrm{J}_{A_m}$ which we also write more simply as $\textrm{J}_{1}$, $\textrm{J}_{2}, \dots, \textrm{J}_{m}$. 
Set
\begin{empheq}[box = \mybluebox]{equation}\label{rea1}
\text{$\mathbf{A}= A_1 \times A_2 \times \dots \times A_{m}$.} 
\end{empheq}
Then
\begin{equation}\label{res111}
\mathbf{\textrm{J}_{A}}: \mathbf{X} \to \mathbf{X} : \big(x_1, x_2, \dots, x_m\big)\mapsto \big( \textrm{J}_{1} x_1, \textrm{J}_{2} x_2, \dots, \textrm{J}_{m} x_m \big).
\end{equation} 
Define the \emph{circular right-shift operator}
\begin{empheq}[box = \mybluebox]{equation}\label{yy7}
\text{$\mathbf{R}: \big( x_1, x_2, \dots, x_m \big) \mapsto \big(x_m, x_1, x_2, \dots, x_{m-1}\big)$.} 
\end{empheq}
Define the \emph{fixed point sets of the cyclic compositions of resolvants}:
\begin{align}
& F_1 := \fix ( \textrm{J}_1 \textrm{J}_{m} \dots \textrm{J}_2 ), \label{f1}\\
& F_2 := \fix ( \textrm{J}_2 \textrm{J}_1 \textrm{J}_m \dots \textrm{J}_3), \\
&  \ \ \ \ \ \vdots  \\
& F_m := \fix ( \textrm{J}_m \textrm{J}_{m-1} \dots \textrm{J}_1).\label{fm}
\end{align}

\begin{definition}\label{cydef}\normalfont\cite[Definition 5.1]{alwadani2021thesis}
Let $z_1 \in F_1$.  Set $z_2:= \textrm{J}_{2} z_1$, $z_3:= \textrm{J}_3 z_2$, $\cdots $, $z_{m-1}:= \textrm{J}_{m-1} z_{m-2},$ and $z_m:= \textrm{J}_{m} z_{m-1}$. The truple $\mathbf{z} = \big(z_1, z_2, \dots, z_m\big) \in  \mathbf{X}$ is called \emph{a cycle}.
\end{definition}
\section{Aim and outline of this paper}
Our main results can be summarized as follows
 \begin{itemize}
 \item \cref{toon1} and  \cref{hygtfre1} sketche the  relationship between the cycles and the fixed point sets of the composition of resolvants. 
  \item The cycles that are defined by the resolvant operators can be formulated in Hilbert product space as a solution to a fixed point equation (see \cref{equationfixed} and \cref{wewewew22}). 
  \item We study the set of classical cycles that are defined by using resolvant operators and the set of classical gap vectors see \cref{drswas}. 
  \item If one of the fixed point sets of composition of resolvents is not empty, then the individuals fixed point sets are equal and their intersection is not empty (see \cref{ggdwe2}). 
 
 \item In \cref{dssesrs}, we use Attouch--Th{\'e}ra duality to study the cycles and the fixed point sets of compositions of resolvents operators.
  \end{itemize}

 \section{Attouch--Th{\'e}ra duality}
Let $A$ and $B$ be two maximally monotone operators on $X$. The \emph{primal problem} associated with $ \big( A, B \big)$ is to
\begin{empheq}[box=\mybluebox]{equation}\label{Sum1}
\text{ find $x \in X$ such that $0 \in Ax + Bx$.}
 \end{empheq}
 The set of \emph{primal solutions} associated with $\big( A,B \big)$ are the solutions to the corresponding sum problem \cref{Sum1} are defined as 
\begin{empheq}[box=\mybluebox]{equation}\label{SumSol}
\text{psol$( A, B):=$ zer$( A, B) = ( A, B)^{-1} (0) = \Big\{ x \in X   \mathrel{\Big|}  0 \in (A+B) x \Big\}$}
 \end{empheq} 
 Now define $B^{\ovee}:= (- \Id ) \circ  B \circ (- \Id)$ and $B^{-\ovee} := (B^{-1})^{\ovee} = (B^{\ovee})^{-1}$. This allows us to define the \emph{dual pair} of $(A, B)$:
 \begin{equation}
 (A, B)^{*} := ( A^{-1}, B^{- \ovee} ).
 \end{equation}
 Then the \emph{dual problem} associated with $(A, B)$ is defined to be the primal problem associated with the  dual pair $( A^{-1}, B^{-\ovee} )$:
 \begin{empheq}[box=\mybluebox]{equation}\label{Sum2}
\text{ find $y \in X$ such that $0 \in A^{-1}y + B^{-\ovee}y = A^{-1} y - B^{-1} (-y)$.}
 \end{empheq}
 The set of of \emph{dual solutions} associated with $\big( A,B \big)$ are the solutions to the corresponding sum problem \cref{Sum2}: 
\begin{empheq}[box=\mybluebox]{equation}\label{SumSol2}
\text{dsol$( A, B):=$ psol$( A, B)^{*}=$ zer $( A^{-1}+  B^{-\ovee})  = \Big\{ y \in X   \mathrel{\Big|}  0 \in (A^{-1}+B^{-\ovee}) y \Big\}.$}
 \end{empheq}  
 Because $(A^{-1})^{-1} = A$, $(A^{\ovee})^{\ovee} = A$, and $(A^{-\ovee})^{-\ovee} = A$, we have 
 \begin{equation}
 (A, B)^{**} = (A, B).
 \end{equation}
 
  \begin{lemma}\normalfont Let $A$ and $B$ be maximally monotone on $X$. Let $x$ and $y$ in $X$. Then the following hold:
  \begin{enumerate}
  \item\label{ton1} If $\text{psol}(A,B) = \{ x\}$, then $$\text{dsol}(A,B) = Ax \cap (-Bx),$$ and $$ \text{dsol}(A,B) = Ax \cap B^{\ovee} (-x) .$$
  \item\label{ton2} If $\text{dsol}(A,B) = \{ y\}$, then $$\text{psol}(A,B) = (A^{-1} y ) \cap B^{-1} (-y),$$ and $$\text{psol}(A, B) =(A^{-1}y) \cap (-B^{-\ovee} (y)) . $$
  \item\label{ton3} If $\text{psol}(A,B) = \{ x\}$ and $Ax$ is a singelton, then $ \text{dsol}(A,B) = Ax.$
  \item\label{ton4} If $\text{psol}(A,B) = \{ x\}$ and $B x$ and $B^{\ovee} (-x)$ are singelton, then $$ \text{dsol}(A,B) = -Bx,$$ and  $$ \text{dsol}(A,B) = B^{\ovee} (-x) . $$
  \item\label{ton5} If $\text{dsol}(A,B) = \{ y\}$ and $A^{-1}y$ is a singelton, then $ \text{psol}(A,B) = A^{-1} y$.
  \item\label{ton6} If $\text{dsol}(A,B) = \{ y\}$ and $B^{-1} (-y)$ and $(- B^{-\ovee} (y))$ are  a singelton, then $$ \text{psol}(A,B) = B^{-1} (-y),$$ 
  and 
 $$   \text{psol}(A,B) = ( - B^{-\ovee} (y) ) .$$
  \end{enumerate}
  \end{lemma}
  
  \begin{proof}
  \ref{ton1}:  From \cref{SumSol}, we have  \begin{align*}
 x \in \text{psol}(A, B) & \Leftrightarrow (A+ B)^{-1}  (0) \neq \emptyset \\
 & \Leftrightarrow  \emptyset \neq Ax \cap (- B x)  \\
 & \Leftrightarrow  \emptyset \neq  Ax \cap \Big( (- \Id)  \circ B \circ ( -  \Id) (-x)\Big)  \\
 &  \Leftrightarrow  \emptyset \neq  Ax \cap B^{\ovee} (-x)  \\
 &  \Leftrightarrow  \emptyset \neq Ax \cap B^{\ovee} (-x)  \subseteq \text{dsol}(A, B). \\
 &   \Leftrightarrow  \emptyset \neq Ax \cap (- B x) \subseteq \text{dsol}(A, B).
 \end{align*}
 Since $\text{psol}(A,B) = \{ x\}$ and by using \cref{SumSol}, we obtain $  Ax \cap (- B x)  = \text{dsol}(A, B), $
 and 
 $ Ax \cap B^{\ovee} (-x) =  \text{dsol}(A, B). $ \\
  \ref{ton2}: From \cref{SumSol2}, we have \begin{align*}
 y \in \text{dsol}(A, B) & \Leftrightarrow (A^{-1}+ B^{- \ovee})^{-1}  (0) \neq \emptyset \\
 & \Leftrightarrow A^{-1} (y) \cap (- B^{-\ovee} (y)) \neq \emptyset \\
 & \Leftrightarrow A^{-1} (y) \cap B^{-1} (-y) \neq  \emptyset \\
 &  \Leftrightarrow  \emptyset \neq A^{-1} (y) \cap B^{-1} (-y)  \subseteq \text{psol}(A, B).
 \end{align*}
 Since $\text{dsol}(A,B) = \{ y\}$ and by using \cref{SumSol2}, we obtain 
 $$  A^{-1} (y) \cap B^{-1} (-y)  = \text{psol}(A, B), $$
  \ref{ton3}: From \cref{ton1}, we have 
  $ Ax \cap (- B x)  = \text{dsol}(A, B),$
  and 
  $  Ax \cap B^{\ovee} (-x) =  \text{dsol}(A, B). $
  Since  $Ax$ is a singelton then we obtain
  $$  \text{dsol}(A, B) =  Ax \cap (- B x) = Ax,  $$
  and 
  $$ \text{dsol}(A, B)=  Ax \cap B^{\ovee} (-x) = Ax. $$
  \ref{ton4}: From \cref{ton1}, we have 
  $ Ax \cap (- B x)  = \text{dsol}(A, B),$
  and 
  $  Ax \cap B^{\ovee} (-x) =  \text{dsol}(A, B). $
  Since  $B x$ and $B^{\ovee}(-x) $ are singelton then we obtain
  $$  \text{dsol}(A, B) =  Ax \cap (- B x) = -B x,  $$
  and 
  $$ \text{dsol}(A, B)=  Ax \cap B^{\ovee} (-x) = B^{\ovee} (-x) . $$
  \ref{ton5}: From \cref{ton2}, we have 
  $ (A^{-1}y) \cap B^{-1} (-y)  = \text{psol}(A, B),$
  and 
  $   (A^{-1}y) \cap (-B^{-\ovee} (y)) =  \text{dsol}(A, B). $
  Since  $A^{-1} y$ is a singelton then we obtain
  $$  \text{psol}(A, B) =(A^{-1}y) \cap B^{-1} (-y) = A^{-1} y,  $$
  and 
  $$ \text{psol}(A, B) =(A^{-1}y) \cap (-B^{-\ovee} (y)) =  A^{-1} y. $$
  \ref{ton6}: From \cref{ton2}, we have 
  $ (A^{-1}y) \cap B^{-1} (-y)  = \text{psol}(A, B),$
  and 
  $   (A^{-1}y) \cap (-B^{-\ovee} (y)) =  \text{dsol}(A, B). $
  Since  $B^{-1} (-y) $ and $(-B^{-\ovee} (y))$  are  singelton then we obtain
  $$  \text{psol}(A, B) =(A^{-1}y) \cap B^{-1} (-y) =  B^{-1} (-y),  $$
  and 
  $$ \text{psol}(A, B) =(A^{-1}y) \cap (-B^{-\ovee} (y)) =  -B^{-\ovee} (y). $$
\end{proof}\\
For more information about the Attouch--Th{\'e}ra duality, we refer the reader to \cite{attouch1996general}.

\section{Correspondence of Properties and Results}
\begin{theorem}\label{toon1}\normalfont  The following are equivalent
\begin{enumerate}
\item\label{toon101} Cycle exists. 
\item\label{toon102} The fixed point sets of cyclic compositions of resolvants $F_i \neq \emptyset$ for all $1 \leq i \leq m$.
\end{enumerate}
\end{theorem}
\begin{proof}
" \ref{toon101}$\Rightarrow$\ref{toon102}": Let $\mathbf{z}= \big(  z_1, z_2, \dots, z_{m-1}, z_m \big)$ be a cycle. Then by \cref{cydef}, we have $z_1 = \textrm{J}_{1} z_m$, $z_2= \textrm{J}_{2} z_1$, $z_3= \textrm{J}_3 z_2$, $\cdots $, $z_{m-1}= \textrm{J}_{m-1} z_{m-2},$ and $z_m= \textrm{J}_{m} z_{m-1}$. This gives that 
\begin{align*} 
& z_1 =  \textrm{J}_{1}  \textrm{J}_{m} \dots  \textrm{J}_{3}  \textrm{J}_{2} z_1 \\
& z_2 =  \textrm{J}_{2}  \textrm{J}_{1} \dots  \textrm{J}_{4}  \textrm{J}_{3} z_2 \\
& \ \ \ \ \ \ \ \ \ \ \ \vdots \\
& z_i =  \textrm{J}_{i}  \textrm{J}_{i-1} \dots  \textrm{J}_{1}  \textrm{J}_{m}  \dots \textrm{J}_{i+1} z_i \\
& \ \ \ \ \ \ \ \ \ \ \ \vdots \\
& z_m =  \textrm{J}_{m}  \textrm{J}_{m-1} \dots  \textrm{J}_{2}  \textrm{J}_{1} z_m.
\end{align*}
Therefore,  $z_1 \in F_1$, $z_2 \in F_2$, $\dots$, $z_i \in F_i$, $\dots$, $z_m \in F_m$ by \cref{f1}-\cref{fm}. This implies that $F_1 \neq \emptyset$, $F_2 \neq \emptyset$, $\dots$, $F_i \neq \emptyset$, $\dots$, $F_m \neq \emptyset$.\\
" \ref{toon102}$\Rightarrow$\ref{toon101}": Let $F_m \neq \emptyset$ and $z_m \in F_m$. Then 
$$ z_m \in \fix \big(  \textrm{J}_{m}  \textrm{J}_{m-1} \dots  \textrm{J}_{2}  \textrm{J}_{1} \big)   z_m,$$
by \cref{fm}. Therefore, 
$$  z_m =   \textrm{J}_{m}  \textrm{J}_{m-1} \dots  \textrm{J}_{2}  \textrm{J}_{1} z_m $$
Next, Applying $ \textrm{J}_1$ we obtain 
$$ \textrm{J}_1  z_m = \textrm{J}_1  \big(  \textrm{J}_{m}  \textrm{J}_{m-1} \dots  \textrm{J}_{2} \big) \big( \textrm{J}_{1} z_m \big), $$
which is equivalent to 
$ \textrm{J}_1  z_m \in \fix \big( \textrm{J}_{1}  \textrm{J}_{m} \textrm{J}_{m-1} \dots  \textrm{J}_{3}  \textrm{J}_{2} \big) \Leftrightarrow  \textrm{J}_1 z_m \in F_1 \neq \emptyset $. Additionaly, let $z_2 \in F_2$ such taht $z_2 = \textrm{J}_{2} z_1$. Keep doing this we obtain, 
$$ z_{m-2} =  \textrm{J}_{m-2}  \textrm{J}_{m-3}  \dots  \textrm{J}_{1}  \textrm{J}_{m} \textrm{J}_{m-1} z_{m-2}.$$
Then, 
$$   \textrm{J}_{m-1}  z_{m-2} =  \textrm{J}_{m-1} \big( \textrm{J}_{m-2}  \textrm{J}_{m-3}  \dots  \textrm{J}_{1}  \textrm{J}_{m} \big) \big( \textrm{J}_{m-1} z_{m-2} \big),$$
and 
$  \textrm{J}_{m-1}  z_{m-2}  \in \fix \big(   \textrm{J}_{m-1}  \textrm{J}_{m-2}  \textrm{J}_{m-3}  \dots  \textrm{J}_{1}  \textrm{J}_{m} \big)$, which is equivalent to $ \textrm{J}_{m-1}  \textrm{J}_{m-2} \in F_{m-1} \neq \emptyset$. Moreover, let $z_{m-1} \in F_{m-1}$ such that $z_{m-1} = \textrm{J}_{m-1} z_{m-2}$. Then we obtain 
$$   z_{m-1} =  \textrm{J}_{m-1}  \textrm{J}_{m-2}   \textrm{J}_{m-3}\dots  \textrm{J}_{1}  \textrm{J}_{m}  z_{m-1},$$
and 
$$  \textrm{J}_{m}   z_{m-1} =   \textrm{J}_{m}  \big( \textrm{J}_{m-1}  \textrm{J}_{m-2}   \textrm{J}_{m-3}\dots  \textrm{J}_{1} \big)  \textrm{J}_{m}  z_{m-1}. $$
Therefore, 
$$  \textrm{J}_{m}   z_{m-1} \in \fix   \big( \textrm{J}_{m}  \textrm{J}_{m-1} \textrm{J}_{m-2} \textrm{J}_{m-3} \dots  \textrm{J}_{1} \big) . $$
This is equivalent to $  \textrm{J}_{m}   z_{m-1} \in \fix  F_m \neq \emptyset$. All these together give $\big(  z_1, z_2, \dots,  z_{m-1}, z_m \big) \in \mathbf{X}$ satisfying that 
$$  \big(  z_1, z_2,  \dots,  z_{m-1}, z_m \big)  = \big( \textrm{J}_{1} z_m, \textrm{J}_{2} z_1, \dots, \textrm{J}_{m-1} z_{m-2}, \textrm{J}_{m} z_{m-1}\big) .$$

\end{proof}
\begin{lemma}\label{equationfixed}\normalfont
Let $\mathbf{z} \in \mathbf{X}$ is a cycle. Then 
\begin{equation}\label{ddq}
\mathbf{z} = \mathbf{J}_{\mathbf{A}} \big( \mathbf{Rz} \big),
\end{equation}
and solving \cref{ddq} is equivalent to solve
\begin{equation}
0 \in \mathbf{A} (\mathbf{z}) + \big( \mathbf{\Id} - \mathbf{R}\big) (\mathbf{z}).
\end{equation}
\end{lemma}

\begin{proof}
Given that $\mathbf{z} = \big( z_1, z_2, \dots, z_m  \big) \in \mathbf{X}$ is a cycle. Then  \cref{cydef} gives that $z_1 = \textrm{J}_{1} z_m$, $z_2= \textrm{J}_{2} z_1$, $z_3= \textrm{J}_3 z_2$, $\cdots $, $z_{m-1}= \textrm{J}_{m-1} z_{m-2},$ and $z_m= \textrm{J}_{m} z_{m-1}$. Hence, 
\begin{align*}
\mathbf{z}= \big( z_1, z_2, \dots, z_{m-1} , z_m\big)& = \big(  \textrm{J}_{1} z_m,  \textrm{J}_{2} z_1, \dots,  \textrm{J}_{m-1} z_{m-2}, \textrm{J}_{m} z_{m-1} \big) \\
& = \big( \textrm{J}_{1}, \textrm{J}_{2}, \dots, \textrm{J}_{m-1}, \textrm{J}_{m} \big) \big(  z_m, z_1, \dots, z_{m-2}, z_{m-1}\big) \\
& = \mathbf{J}_{\mathbf{A}} \big(   \mathbf{R} \big( z_1, z_2, \dots, z_{m-1}, z_{m} \big)\big) \\
& = \mathbf{J}_{\mathbf{A}} \big(   \mathbf{R z} \big).
\end{align*}
Note that $ \mathbf{z} = \mathbf{J}_{\mathbf{A}} (\mathbf{Rz}) \Leftrightarrow \mathbf{z} = \big(  \mathbf{Id} + \mathbf{A}\big)^{-1} ( \mathbf{Rz})$ by \cref{quli1}. Therefore, we obtain 
$$  \mathbf{Rz} \in \mathbf{z} + \mathbf{A} (\mathbf{z}) \Leftrightarrow 0 \in \mathbf{A} ( \mathbf{z}) + \big(  \mathbf{Id} - \mathbf{R}\big)  ( \mathbf{z}).$$
\end{proof}

Define the set of all cycles by 
\begin{empheq}[box=\mybluebox]{equation}\label{nnote11}
\text{ $\mathbf{Z}:= \fix (\mathbf{J}_{\mathbf{A}} \mathbf{R}).$ }
 \end{empheq}
 
Define
\begin{equation}
F_{i}:=\big\{ z \in X \mid  z = \textrm{J}_{i} \dots \textrm{J}_{1} \textrm{J}_{m} \dots \textrm{J}_{i+1} z \big\}. 
\end{equation}
Moreover,
\begin{equation}
Q_i: \mathbf{X} \to X: \mathbf{z} \mapsto z_i.
\end{equation}	
The relationship between the fixed point set of composition of $m$ resolvants $F_i$'s and the set of all cycles $\mathbf{Z}$ are given in the following theorem. 
\begin{theorem}\label{hygtfre1}\normalfont For every $1 \leq i, j \leq m$, the following hold;
\begin{enumerate}
\item\label{ff1} $F_i$	are closed and convex. Moreover, 
\begin{small}
\begin{align}
& F_{m}  = \big(  \textrm{J}_{m}  \textrm{J}_{m-1} \dots  \textrm{J}_{3}  \textrm{J}_{2}\big) \big( F_1 \big) =  \big(  \textrm{J}_{m}  \textrm{J}_{m-1} \dots  \textrm{J}_{3}  \big) \big( F_2 \big) = \dots = \textrm{J}_{m}  \textrm{J}_{m-1} \big( F_{m -2}\big) = \textrm{J}_{m} \big( F_{m-1}\big). \\
& F_{m-1} = \big( \textrm{J}_{m-1} \dots  \textrm{J}_{3} \textrm{J}_{2} \textrm{J}_{1} \big) \big( F_m \big) = \big( \textrm{J}_{m-1} \dots  \textrm{J}_{3} \textrm{J}_{2} \big) \big( F_1 \big) = \dots = \textrm{J}_{m-1} \big( F_{m-2}\big).\label{991}\\
& \vdots \\
& F_2 = \big( \textrm{J}_{2} \textrm{J}_{1} \textrm{J}_{m} \dots \textrm{J}_{4}\big) \big( F_3\big) = \big( \textrm{J}_{2} \textrm{J}_{1} \textrm{J}_{m} \dots \textrm{J}_{5}\big) \big( F_4\big) = \dots =  \textrm{J}_{2} \textrm{J}_{1} \big( F_m  \big) = \textrm{J}_{2} \big( F_1\big). \\
& F_1 = \big( \textrm{J}_{1} \textrm{J}_{m} \dots  \textrm{J}_{4}  \textrm{J}_{3} \big) \big( F_2\big) = \big( \textrm{J}_{1} \textrm{J}_{m} \dots  \textrm{J}_{4} \big) \big( F_3 \big) = \dots = \big( \textrm{J}_{1} \textrm{J}_{m} \big) \big( F_{m-1}\big) =  \textrm{J}_{1} \big( F_m\big).\label{992} 
\end{align}
\end{small}
\item\label{dakhel1} $\cap^{m}_{i =1} \Fix \textrm{J}_{i} \subseteq \cap^{m}_{i =1}  F_{i} $ and if $F_i = \varnothing$, then $\cap^{m}_{i =1} \Fix \textrm{J}_{i} =  \varnothing$. 
\item\label{ff2} For $ 1\leq i \leq m-1$, $\textrm{J}_{i+1} \big( F_i\big) = \big(F_{i+1}\big)$, and $\textrm{J}_{1} \big( F_{m}\big) = F_{1} $. Therefore, 
\begin{equation}
\mathbf{\textrm{J}_{A}} \mathbf{R} \big( F_1 \times F_2 \times \dots \times F_m\big) = F_1 \times F_2 \times \dots \times F_m.
\end{equation}
\item\label{yy1} $F_i \neq \varnothing$ if and only if $F_j \neq \varnothing$ if and only if $\mathbf{Z} = \varnothing$.
\item\label{yy2} $\mathbf{Z}$ is closed and convex, and $\mathbf{Z} \subseteq F_1 \times F_2 \times \dots \times F_m$.
\item\label{yy3} The mapping $Q_i{|}_{\mathbf{Z}} : \mathbf{Z} \to F_{i}$ is bijective and $Q_{i} \big( \mathbf{Z}\big) = F_i$.
\end{enumerate}		
\end{theorem}
\begin{proof}
\ref{ff1}: Since each $\textrm{J}_{i}$ is firmly nonexpansive, so nonexpansive, then by \cite[Lemma~$2.1.12$~(ii)]{cegielski2012iterative}, the composition 
$$ \textrm{J}_{i} \dots \textrm{J}_{1} \textrm{J}_{m}  \dots \textrm{J}_{i+1}$$
is nonexpansive. Then by \cite[Proposition~$2.1.11$]{cegielski2012iterative} $F_i$ is closed and convex.\\
 Let $x \in F_m \Leftrightarrow x \in \Fix \big(  \textrm{J}_{m}  \textrm{J}_{m-1} \dots  \textrm{J}_{3} \textrm{J}_{2}  \textrm{J}_{1} \big) \Leftrightarrow x = \textrm{J}_{m}  \textrm{J}_{m-1} \dots  \textrm{J}_{3} \textrm{J}_{2}  \textrm{J}_{1} x $. Then, 
$$ \textrm{J}_{1} x =  \textrm{J}_{1} \big( \textrm{J}_{m}  \textrm{J}_{m-1} \dots  \textrm{J}_{3} \textrm{J}_{2}  \textrm{J}_{1} \big) x = \big( \textrm{J}_{1} \textrm{J}_{m}  \textrm{J}_{m-1} \dots  \textrm{J}_{3} \textrm{J}_{2} \big) \big( \textrm{J}_{1} x\big)$$
Therefore, $\textrm{J}_{1} x \in \Fix \big( \textrm{J}_{1} \textrm{J}_{m} \textrm{J}_{m-1} \dots \textrm{J}_{3} \textrm{J}_{2}\big) \Leftrightarrow \textrm{J}_{1} x \in F_1$. It follows that 
\begin{equation}
\textrm{J}_{1} \big( F_m\big) \subseteq F_1
\end{equation}
Moreover, 
\begin{align}
\big( \textrm{J}_{2} \textrm{J}_{1} \big) \big( F_m \big) = \big( \textrm{J}_{2} \textrm{J}_{1} \big) \big( \Fix \big( \textrm{J}_{m} \textrm{J}_{m-1} \dots \textrm{J}_{2} \textrm{J}_{1} \big)\big) &\subseteq  \textrm{J}_{2} \big( \Fix \big(  \textrm{J}_{1}  \textrm{J}_{m}  \textrm{J}_{m-1} \dots  \textrm{J}_{3}  \textrm{J}_{2} \big)\big)\\
& \subseteq  \Fix \big( \textrm{J}_{2} \textrm{J}_{1} \textrm{J}_{m} \textrm{J}_{m-1} \dots \textrm{J}_{4} \textrm{J}_{3}\big) = F_{2},
\end{align}	
hence
\begin{small}
\begin{align}
\big( \textrm{J}_{3} \textrm{J}_{2}\textrm{J}_{1} \big) \big( F_m \big) = \big( \textrm{J}_{3} \textrm{J}_{2}\textrm{J}_{1} \big) \big( \Fix \big( \textrm{J}_{m} \textrm{J}_{m-1} \dots \textrm{J}_{2} \textrm{J}_{1} \big) \big) &\subseteq \big( \textrm{J}_{3} \textrm{J}_{2} \big) \big( \Fix \big( \textrm{J}_{1} \textrm{J}_{m} \textrm{J}_{m-1} \dots \textrm{J}_{3} \textrm{J}_{2}\big)\big) \\
& \subseteq  \textrm{J}_{3} \big( \Fix \big( \textrm{J}_{2}\textrm{J}_{1} \textrm{J}_{m} \textrm{J}_{m-1} \dots \textrm{J}_{4} \textrm{J}_{3} \big)\big) \\
& \subseteq \Fix \big(  \textrm{J}_{3}  \textrm{J}_{2}  \textrm{J}_{1}  \textrm{J}_{m}  \textrm{J}_{m-1} \dots  \textrm{J}_{5}  \textrm{J}_{4}\big) = F_{3},	
\end{align}
\end{small}
until finally
\begin{small}
\begin{align}\label{12}
F_{m}  = \Fix \big( \textrm{J}_{m} \textrm{J}_{m-1} \dots \textrm{J}_{2}\textrm{J}_{1} & = \big( \textrm{J}_{m} \textrm{J}_{m-1} \dots \textrm{J}_{2}\textrm{J}_{1} \big) \big(\Fix \big( \textrm{J}_{m} \textrm{J}_{m-1} \dots \textrm{J}_{2}\textrm{J}_{1} \big)\big) \\
& \subseteq \big( \textrm{J}_{m} \textrm{J}_{m-1} \dots \textrm{J}_{2} \big) \big( \Fix \big(  \textrm{J}_{1} \textrm{J}_{m} \dots \textrm{J}_{3} \textrm{J}_{2} \big)\big) \\
& \ \ \vdots \\
& \subseteq \textrm{J}_{m} \big( \Fix \big(  \textrm{J}_{m-1}  \textrm{J}_{m-2} \dots  \textrm{J}_{2}  \textrm{J}_{1}  \textrm{J}_{m}\big)\big) =  \textrm{J}_{m} \big( F_{m-1}\big) \\
& \subseteq \Fix \big(  \textrm{J}_{m}  \textrm{J}_{m-1} \dots  \textrm{J}_{2}  \textrm{J}_{1}\big) = F_m.\label{hh}
\end{align}	
\end{small}
Hence, equality holds throughout \cref{12}-\cref{hh} and we are done. The same approach will verify \cref{991}-\cref{992}.\\
\ref{dakhel1}: It is well known that $\cap^{m}_{i =1} \Fix \textrm{J}_{i} \subseteq F_{1}, \ \ \cap^{m}_{i =1} \Fix \textrm{J}_{i} \subseteq F_{2}, \ \ \cdots, \ \ \cap^{m}_{i =1} \Fix \textrm{J}_{i} \subseteq F_{m}$. Hence $$\bigcap^{m}_{i =1} \Fix \textrm{J}_{i} \subseteq \bigcap^{m}_{i =1} F_{i}.$$
This also implies that $\cap^{m}_{i =1} \Fix \textrm{J}_{i} = \varnothing$ if $F_i = \varnothing$.
\ref{ff2}: From \ref{ff1}, we have 
\begin{equation}\label{fac1}
\textrm{J}_{1} \big( F_m\big) = F_1, \ \ \ \ \ \textrm{J}_{2} \big( F_1\big) = F_2, \ \ \ \ \ \dots, \ \ \ \ \ \textrm{J}_{m} \big( F_{m-1}\big) = F_m .
\end{equation}
Using \cref{res111}, \cref{yy7} and \cref{fac1}, we obtain
\begin{align*}
\mathbf{\textrm{J}_{A}} \mathbf{R} \big( F_1 \times F_2 \times \dots \times F_{m-1} \times F_m\big) &=   \mathbf{\textrm{J}_{A}} \big(  F_m \times F_1 \times F_2 \times \dots \times F_{m-1} \big)\\
& = \big(\textrm{J}_{1}, \textrm{J}_{2}, \cdots, \textrm{J}_{m} \big) \big(  F_m \times F_1 \times F_2 \times \dots \times F_{m-1} \big) \\
& = F_1 \times F_2 \times \dots \times F_{m-1} \times F_m.
\end{align*}  
\ref{yy1}: Clear from the definition of $F_i$, $F_j$ and $\mathbf{Z}$. \\
\ref{yy2}: Since $\mathbf{\textrm{J}_{A}} \mathbf{R}$ is nonexpansive and because $\mathbf{Z} = \Fix \mathbf{\textrm{J}_{A}} \mathbf{R} $, by \cite[Proposition~$2.1.11$]{cegielski2012iterative}, $\mathbf{Z}$ is closed and convex. Moreover, let $\mathbf{z}= \big( z_1, z_2, \dots, z_m \big) \in \Fix \mathbf{\textrm{J}_{A}} \mathbf{R}  \Leftrightarrow \mathbf{z} = \mathbf{\textrm{J}_{A}} \mathbf{R} \mathbf{z} $.  This implies 
\begin{align*}
& z_1 =  \textrm{J}_{1}  \textrm{J}_{m} \dots  \textrm{J}_{2} z_1, \\
& \ \ \ \ \ \ \ \ \ \vdots \\
& z_{i} =  \textrm{J}_{i} \textrm{J}_{i- 1} \dots \textrm{J}_{1}  \textrm{J}_{m}  \dots  \textrm{J}_{i+1} z_i,\\
& \ \ \ \ \ \ \ \ \ \vdots \\
& z_m =  \textrm{J}_{m}  \textrm{J}_{m-1} \dots  \textrm{J}_{1} z_m.
\end{align*}
Hence, $\mathbf{z} = \big( z_1, z_2, \dots, z_m \big) \in F_1 \times F_2 \times \dots \times F_{m-1} \times F_{m}$. Since this is true for all $ \mathbf{z} \in \mathbf{Z}$, therefore, $$\mathbf{Z} \subseteq F_1 \times F_2 \times \dots \times F_{m-1} \times F_m.$$  
\ref{yy3}: It is clear that \ref{ff1} that $Q_{i}: \mathbf{Z} \to F_i$ is surjective. To show $Q_i$ is injective, suppose $\mathbf{z} = \big( z_1, z_2, \dots, z_m\big)$, $\mathbf{\widetilde{z}} = \big(\widetilde{z}_1, \widetilde{z}_2, \dots, \widetilde{z}_m \big) \in \mathbf{Z}$ and $Q_{i} \big( \mathbf{z} \big) = Q_i \big( \mathbf{\widetilde{z}} \big) \Rightarrow z_i = \widetilde{z}_i $. Because $\mathbf{z}$ and $ \mathbf{\widetilde{z}}$ are cycles then we obtain
\begin{align}\label{ppw}
& z_{i+1} = \textrm{J}_{i+1} z_i = \textrm{J}_{i+1} \widetilde{z}_i= \widetilde{z}_{i+1} \\
& \ \ \ \ \ \ \ \vdots \\
& z_{m} = \textrm{J}_{m} z_{m-1} = \textrm{J}_{m} \widetilde{z}_{m-1}= \widetilde{z}_{m} \\
&  z_{1} = \textrm{J}_{1} z_{m} = \textrm{J}_{1} \widetilde{z}_{m}= \widetilde{z}_{1} \\
&  z_{2} = \textrm{J}_{2} z_{1} = \textrm{J}_{2} \widetilde{z}_{1}= \widetilde{z}_{2} \\
& \ \ \ \ \ \ \ \vdots \\ 
&  z_{i-1} = \textrm{J}_{i-1} z_{i-2} = \textrm{J}_{i-1} \widetilde{z}_{i-2}= \widetilde{z}_{i-1},\label{ttt} 
\end{align}
hence from \cref{ppw}-\cref{ttt}, we have 
$$ \mathbf{z} =  \mathbf{\widetilde{z}} . $$
\end{proof}

\begin{lemma}\label{ggggtsd}\normalfont Let $ \cap^{m}_{i = 1} \Fix \textrm{J}_{i} \neq \varnothing:= D$. Then the following are hold:
\begin{enumerate}
\item\label{yyyde1} We have $F_i = D$ for every $1 \leq i \leq m$.
\item\label{yyyde2} $\mathbf{Z} = \{ \big( z, z , \cdots, z\big) \mid z \in D \} = D^{m} \cap \mathbf{\Delta}$.	
\end{enumerate}
\end{lemma}
\begin{proof}
\ref{yyyde1}: $\textrm{J}_i$ is firmly nonexpansive for every $1 \leq i \leq m$, then by \cite[Corollary~$4.51$]{BC2017}, we have 
$$ \big( \forall (1\leq i \leq m)\big), \ \  \Fix \big( \textrm{J}_i \textrm{J}_{i-1} \cdots \textrm{J}_{1}\textrm{J}_{m} \cdots \textrm{J}_{i+1}  \big) = D. $$
\ref{yyyde2}: Let $\mathbf{z} = \big( z_1, z_2, \cdots, z_m\big) \in \mathbf{Z}$. Then, 
\begin{align*}
& z_1 = \textrm{J}_1 \textrm{J}_m \textrm{J}_{m-1} \cdots \textrm{J}_2 z_1 \Leftrightarrow z_1 \in F_{1} = D \\
& z_2 = \textrm{J}_2 \textrm{J}_1 \textrm{J}_m \cdots \textrm{J}_3 z_2 \Leftrightarrow z_2 \in F_2 = D\\
& \vdots \\
&  z_m = \textrm{J}_m \textrm{J}_{m-1} \textrm{J}_{m-2} \cdots \textrm{J}_1 z_m \Leftrightarrow z_m \in F_{m} = D, 
\end{align*} 
hence, $$\mathbf{z} = \big( z_1, z_2, \cdots, z_m\big) \in F_1 \times F_2 \times \cdots \times F_m = D \times U \times D \times \cdots \times D = D^{m}.$$
Moreover, 
 $$\mathbf{z} = \big( z_1, z_2, \cdots, z_m\big) = \big( z, z, \cdots, z\big) \in \mathbf{D}.$$
 Therefore, 
$$ \mathbf{z} \in D^{m} \cap \mathbf{D}. $$ 
\end{proof}

\begin{remark}\label{eeeesw1}\normalfont When $m = 2$, we have 
$  \textrm{J}_{1} \big( F_2\big) = F_1,$ and  $\textrm{J}_{2} \big( F_1\big) = F_2. $	
\end{remark}
\begin{proof}
 Let $z \in F_1$ and $\widetilde{z} \in F_2$, then we have $z = \textrm{J}_{1}  \textrm{J}_{2} z $ and $\textrm{J}_{2} z = \textrm{J}_{2} \textrm{J}_{1} (\textrm{J}_{2} z)$. Therefore,
\begin{equation}\label{tt1f1} 
 \textrm{J}_{2} \big( F_1\big) \subseteq F_{2}
\end{equation}
Moreover, $\widetilde{z} \in F_2$, then we have $\widetilde{z} = \textrm{J}_{2}  \textrm{J}_{1} \widetilde{z} $ and $\textrm{J}_{1} \widetilde{z} = \textrm{J}_{1} \textrm{J}_{2} (\textrm{J}_{1} \widetilde{z})$. 
Thus, 
\begin{equation}\label{tt1f2}
\textrm{J}_{1} \big( F_2\big) \subseteq F_{1}
\end{equation}
Now apply $\textrm{J}_{1}$ and $\textrm{J}_{2}$ to \cref{tt1f1} and \cref{tt1f2} respectively, we get 
$$  F_{1} \subseteq \textrm{J}_{1} \big( F_2\big) \ \ \ \text{and} \ \ \   F_{2} \subseteq \textrm{J}_{2} \big( F_1\big). $$ 
 Hence, we have $ F_{1} =  \textrm{J}_{1} \big( F_2\big)$ and $ F_{2} =  \textrm{J}_{2} \big( F_1\big)$.
\end{proof}

\begin{lemma}\label{wewewew22} \normalfont Recall from \cref{nnote11} that $\mathbf{Z} = \Fix  \mathbf{\textrm{J}_{A}} \mathbf{R}$. Then we have 
	$$ \mathbf{Z} = \Fix  \mathbf{\textrm{J}_{A}} \mathbf{R} = \Fix \mathbf{ \textrm{J}_{\frac{1}{2} \mathbf{A}}} \big( \frac{ \mathbf{Id}+ \mathbf{R}}{2}\big). $$	
\end{lemma}
\begin{proof}
	Let $\mathbf{x} \in \Fix \big( \textrm{J}_{A} \mathbf{R} \big)$, then
	\begin{align*}
	\mathbf{x} = \textrm{J}_{A} \mathbf{R} \mathbf{x} \Leftrightarrow \mathbf{R} \mathbf{x} \in \mathbf{x} + \mathbf{A} \mathbf{x} & \Leftrightarrow  0 \in \big( \mathbf{x} - \mathbf{R} \mathbf{x}\big) + \mathbf{A} ( \mathbf{x} )\\
	&  \Leftrightarrow 0 \in \frac{\big( \mathbf{x} - \mathbf{R} \mathbf{x}\big)}{2} + \frac{\mathbf{A} ( \mathbf{x} )}{2} \\
	&  \Leftrightarrow 0 \in \mathbf{ x} - \Big( \frac{ \mathbf{Id} + \mathbf{R} }{2}\Big) \mathbf{x} +  \frac{\mathbf{A} ( \mathbf{x} )}{2}  \ \ \ \textrm{adding and subtracting $\frac{\mathbf{x}}{2}$} \\
	&   \Leftrightarrow \Big( \frac{ \mathbf{Id} + \mathbf{R} }{2}\Big) \mathbf{x} \in \Big( \mathbf{Id} + \frac{1}{2} \mathbf{A}\Big) \big( \mathbf{x}\big) \\
	&  \Leftrightarrow  \mathbf{x} = \mathbf{ \textrm{J}_{\frac{1}{2} \mathbf{A}}} \Big( \frac{ \mathbf{Id}+ \mathbf{R}}{2}\Big) \big( \mathbf{x}\big) \\
	&  \Leftrightarrow \mathbf{x} \in \Fix \mathbf{ \textrm{J}_{\frac{1}{2} \mathbf{A}}} \Big( \frac{ \mathbf{Id}+ \mathbf{R}}{2}\Big).
	\end{align*}
\end{proof}

\begin{lemma}\label{ggdwe2}\normalfont Suppose that $\Fix  \textrm{J}_i \neq \varnothing$ for each $1 \leq i \leq m$. Then, the following are equivalent 
\begin{enumerate} 
\item\label{iiigf1} $\cap^{m}_{i=1} \Fix  \textrm{J}_i \neq \varnothing$.
\item\label{iiigf22} $F_1 = F_2 = \cdots = F_m \neq \varnothing$. 	
\end{enumerate}
\end{lemma}
\begin{proof}
\ref{iiigf1}: Let $\cap^{m}_{i=1} \Fix  \textrm{J}_i \neq \varnothing \Rightarrow F_1 \neq \varnothing, F_2 \neq \varnothing, \cdots, F_m \neq \varnothing$ and from \cref{ggggtsd}\ref{yyyde1}, we have  $$F_1 = F_2 = \cdots = F_m = \cap^{m}_{i=1} \Fix  \textrm{J}_i.$$
\ref{iiigf22}: Let $F_1 = F_2 = \cdots = F_m \neq \varnothing$. Then, by \cref{hygtfre1}~\ref{yy1} we have $\mathbf{Z} \neq \varnothing$ 
\end{proof}

\section{Consequences of Attouch-Th\'era  duality}\label{dssesrs}
Recall \cref{rea1} that 
$$  \mathbf{A} = A_1 \times A_2 \times \dots \times A_{m}.$$
From now on, suppose that 
\begin{empheq}[box=\mybluebox]{equation}
\text{$\mathbf{A}$ is
maximally monotone on $\mathbf{X}$,}
 \end{empheq}
and 
\begin{empheq}[box=\mybluebox]{equation}
\text{$\mathbf{C}:= \zer \mathbf{A}$ is
not empty. }
 \end{empheq}
Recall \cref{nnote11} that 
$$ \mathbf{Z}:= \fix (\mathbf{J}_{\mathbf{A}} \mathbf{R}).  $$

\begin{proposition} \normalfont The following holds:
\begin{enumerate}
\item\label{ttq1} $\mathbf{Z} = \text{psol} \big( \mathbf{Id} - \mathbf{R} \big)$.
\item\label{ttq2} $\mathbf{A} + \mathbf{Id} - \mathbf{R} $ is maximally monotone.
\item\label{ttq3} $\mathbf{Z}$ is closed and convex.
\end{enumerate}
\end{proposition}

\begin{proof}
\ref{ttq1}: Combine \cref{equationfixed} and \cref{SumSol}.
\ref{ttq2}: Note that $\mathbf{Id} - \mathbf{R}$ is linear, full domain, and maximally monotone by \cite[Theorem 7.1]{alwadani2021thesis}. Moreover, $\mathbf{A}$ is maximally monotone by assumption. Therefore, the sum is maximally monotone by \cite[Corollary 25.5~(i)]{BC2017}
\ref{ttq3}: It follows directly from \cref{ttq1} and \cref{hygtfre1}~\ref{yy2}. 
\end{proof}   

\begin{theorem}\normalfont Recall from \cref{equationfixed}, the primal (Attouch-Th\'era) problem,
$$ 0 \in \mathbf{A} (\mathbf{z}) + ( \mathbf{Id} - \mathbf{R}) (\mathbf{z}),$$
for the pair $ ( \mathbf{A}, \mathbf{Id} - \mathbf{R} )$. Then the Attouch-Th\'era dual problem is 
\begin{empheq}[box=\mybluebox]{equation}\label{fgfds98d}
\text{$ 0 \in \mathbf{A}^{-1} ( \mathbf{y}) + ( \mathbf{Id} - \mathbf{R})^{-1} ( \mathbf{y})$ }   
 \end{empheq}
or, 
\begin{empheq}[box=\mybluebox]{equation}\label{trewq}
\text{$ 0 \in \big( \mathbf{A}^{-1}  + N_{ \mathbf{D}^\perp } \big)( \mathbf{y} ) + \Big( \frac{1}{2} \mathbf{Id} + \mathbf{T} \Big)( \mathbf{y} )$ }
 \end{empheq}
 Moreover, 
 \begin{equation}\label{fdsxwa}
 \text{dsol} ( \mathbf{A}, \mathbf{Id} - \mathbf{R} ) = \zer \Big(  \mathbf{A}^{-1}  + N_{ \mathbf{D}^\perp} +   \frac{1}{2} \mathbf{Id} + \mathbf{T} \Big).
\end{equation}
\end{theorem}

\begin{proof}
The dual pair of $\big(\mathbf{A}, \big(  \mathbf{Id} - \mathbf{R}\big)\big)$ is $$\big(\mathbf{A}, \big(  \mathbf{Id} - \mathbf{R}\big)\big)^{*} = \big(  \mathbf{A}^{-1}, \big( \mathbf{Id} - \mathbf{R}\big)^{- \ovee} \big).$$
Because of the linearity of $\mathbf{R}$, we have
$$ \big( \mathbf{Id} - \mathbf{R}\big)^{- \ovee} = \big( - \mathbf{Id}\big) \circ \big( \mathbf{Id} - \mathbf{R}\big)^{-1}\circ \big( \mathbf{- Id} \big) = \big( \mathbf{Id} - \mathbf{R}\big)^{-1}. $$
Hence, Attouch-Th\'era dual problem simplifies to 
\begin{equation*}\label{yyydgve2}
0 \in  \mathbf{A}^{-1} \big( \mathbf{y}\big) + \big( \mathbf{Id} - \mathbf{R}\big)^{-1} \big( \mathbf{y}\big).
\end{equation*}
From \cite[Theorem 2.8~(i)]{alwadani2024additiona}, we have 
\begin{align*}
0 \in  \mathbf{A}^{-1} \big( \mathbf{y}\big) + \big( \mathbf{Id} - \mathbf{R}\big)^{-1} \big( \mathbf{y}\big)  & \Leftrightarrow 0 \in \mathbf{A}^{-1} \big( \mathbf{y}\big) + \Big(  \mathbf{N}_{\mathbf{D}^{\perp}}  + \frac{1}{2} \mathbf{Id} + \mathbf{T} \Big) \big( \mathbf{y}\big) \\
& \Leftrightarrow 0 \in \big( \mathbf{A}^{-1} +   \mathbf{N}_{\mathbf{D}^{\perp}}   \big)   (\mathbf{y})+ \Big( \frac{1}{2} \mathbf{Id} + \mathbf{T} \Big)  (\mathbf{y}), 
\end{align*}
which verifies \cref{trewq}. Next, by \cref{SumSol2}, \cref{fgfds98d}, and \cref{trewq}, we have 
\begin{align*}
 \text{dsol} ( \mathbf{A}, \mathbf{Id} - \mathbf{R} ) & = \zer \Big(   \mathbf{A}^{-1} +  \big( \mathbf{Id} - \mathbf{R} \big)^{-1}\Big) \\
 & =  \zer \Big( \mathbf{A}^{-1}  +  \mathbf{N}_{\mathbf{D}^{\perp}}  + \frac{1}{2} \mathbf{Id} + \mathbf{T} \Big).
\end{align*}
\end{proof}

\begin{proposition}\label{hgsttttq}\normalfont The solution set of \cref{fgfds98d} is at most a at most a sigleton and possibly empty.
\end{proposition}
\begin{proof}
See \cite[Proposition 9.2]{alwadani2021thesis}.
\end{proof}
  
\begin{theorem}\label{drswas}\normalfont
Let $  \text{psol} ( \mathbf{A}, \mathbf{Id} - \mathbf{R} )   = \mathbf{Z} $ and recall from \cref{fdsxwa} that 
$$   \text{dsol} ( \mathbf{A}, \mathbf{Id} - \mathbf{R} )  =   \zer \Big( \mathbf{A}^{-1}  +  \mathbf{N}_{\mathbf{D}^{\perp}}  + \frac{1}{2} \mathbf{Id} + \mathbf{T} \Big).$$
Then 
\begin{equation}
 \text{dsol} \big( \mathbf{Id} - \mathbf{R} \big) = \big(  \mathbf{R}- \mathbf{Id}\big) \mathbf{Z} = \begin{cases}
\Big\{  \mathbf{J}_{2 (\mathbf{A}^{-1}  +  \mathbf{N}_{\mathbf{D}^{\perp}}  + \mathbf{T})} (\mathbf{0}) \Big\} , if \ \ \mathbf{Z} \neq \emptyset\\
\emptyset, \ \ \ \ \ \ \ \ \ \ \ \ \ \ \ \ \ \ \ \ \ \ \ \ \ \ \ \ \  \ \ \ \ if \ \  \mathbf{Z} = \emptyset
\end{cases}
\end{equation}
Moreover, if $\mathbf{y}^{*}:= \mathbf{J}_{2 (\mathbf{A}^{-1}  +  \mathbf{N}_{\mathbf{D}^{\perp}}  + \mathbf{T})} (\mathbf{0}) $ exists, then the following are hold:
\begin{enumerate}
\item\label{leen1} $\mathbf{y}^{*} \in \mathbf{D}^{\perp}$.
\item\label{leen2} $\mathbf{y}^{*}  $ is the only vector that makes 
$  \mathbf{A}^{-1} \mathbf{y} \cap - ( \mathbf{N}_{\mathbf{D}^{\perp}} \mathbf{y} + \frac{1}{2} \mathbf{y} + \mathbf{T} \mathbf{y}), $
nonempty
\item\label{leen3} $ \mathbf{Z}= \mathbf{A}^{-1} \mathbf{y}^{*} \cap ( - \frac{1}{2} \mathbf{y}^{*}  - \mathbf{T} \mathbf{y}^{*} - \mathbf{D} )$
\end{enumerate}
\end{theorem}

\begin{proof}
By using \cref{SumSol2}, we have $\mathbf{y} \in \text{dsol} \big( \mathbf{A}, \mathbf{Id} - \mathbf{R} \big)$ gives 
\begin{align*}
& \ \ \ \ \ \ \mathbf{0} \in \mathbf{A}^{-1} (\mathbf{y}) + ( \mathbf{Id} -  \mathbf{R})^{-1} (\mathbf{y}) \ \ \ ( \forall \mathbf{z} \in \mathbf{Z})\\
& \Leftrightarrow \mathbf{z} \in  \mathbf{A}^{-1} (\mathbf{y}) \ \ \text{and} \ \ -  \mathbf{z} \in ( \mathbf{Id} -  \mathbf{R})^{-1} (\mathbf{y}) \ \ \ ( \forall \mathbf{z} \in \mathbf{Z})\\
&  \Leftrightarrow  \mathbf{y} \in \mathbf{A} ( \mathbf{z}) \ \ \text{and} \ \  \mathbf{y} = (  \mathbf{Id} -  \mathbf{R} ) ( - \mathbf{z}) \ \ \ ( \forall \mathbf{z} \in \mathbf{Z}),
\end{align*}
hence, for all $\mathbf{z} \in \mathbf{Z}$, we obtain 
$$  \mathbf{y} = \mathbf{Rz} - \mathbf{z}$$
and 
\begin{align*}
\text{dsol} \big( \mathbf{Id} - \mathbf{R} \big) & = \cup_{\mathbf{z} \in \mathbf{Z}} \{ \mathbf{R z} - \mathbf{z} \mid \mathbf{z} \in \mathbf{Z}\} \\
& = \big(  \mathbf{R} - \mathbf{Id} \big) \mathbf{Z}.
\end{align*}
Additionally,  if $\mathbf{Z} \neq \emptyset$, then using \cite[Theorem 2.4]{alwadani2024additiona} and \cref{quli1} gives 
\begin{align*}
& \ \ \ \ \ \ \ \ \mathbf{0} \in \mathbf{A}^{-1} (\mathbf{y}) + ( \mathbf{Id} - \mathbf{R})^{-1} ( \mathbf{y}) \\
& \Leftrightarrow \mathbf{0} \in \mathbf{A}^{-1} (\mathbf{y}) + \Big(  \frac{1}{2} \mathbf{Id} + \mathbf{T} + \mathbf{N}_{\mathbf{D}^{\perp}}\Big) ( \mathbf{y}) \\
& \Leftrightarrow \mathbf{0} \in \Big(   \mathbf{Id} + 2 \Big( \mathbf{A}^{-1} + \mathbf{T} + \mathbf{N}_{\mathbf{D}^{\perp}} \Big)\Big) ( \mathbf{y}) \\
&  \Leftrightarrow \mathbf{y} = \Big(   \mathbf{Id} + 2 \Big( \mathbf{A}^{-1} + \mathbf{T} + \mathbf{N}_{\mathbf{D}^{\perp}} \Big)\Big)^{-1} ( \mathbf{0}) \\
&  \Leftrightarrow \mathbf{y} =\mathbf{J}_{2( \mathbf{A}^{-1} + \mathbf{T} + \mathbf{N}_{\mathbf{D}^{\perp}} )} ( \mathbf{0})
\end{align*}
However, if $\mathbf{Z} = \emptyset$, then using \cite[Proposition 2.4~(v)]{bauschke2012attouch}
$$ \emptyset =  \text{dsol} \big(  \mathbf{A}, \mathbf{Id}- \mathbf{R}\big)= \text{dsol} \big( \mathbf{Id} - \mathbf{R} \big).  $$
\ref{leen1}: By \cite[Theorem 2.7]{alwadani2024additiona}, we have $ \dom ( \mathbf{Id} - \mathbf{R})^{-1} = \mathbf{D}^{\perp} $. This implies that $$\mathbf{y} \in \mathbf{D}^{\perp}.$$ \ref{leen2}: Combine \cref{hgsttttq} and \cite[Proposition~2.4]{bauschke2012attouch}.\\
\ref{leen3}: Combine \cref{leen1}, \cref{leen2}, and  \cite[Proposition 9.3 (i)]{alwadani2021thesis} where $\mathbf{N}^{-1}_{\mathbf{C}} $ is replaced by $\mathbf{A}^{-1}$.
\end{proof}

\begin{lemma}\label{tresdrd2}\normalfont Denote by $\mathbf{y}^{*} = \big( y_1, y_2, \cdots, y_m\big)$ the unique solution of \cref{trewq}. Then the following hold:
\begin{enumerate}
\item\label{yyy1} The mapping $\textrm{J}_1 : F_m \to F_1$ is bijective on $F_m$ and it is given by $\textrm{J}_1 \big( z\big) = z - y_{1}$. Moreover, for $1 \leq i \leq m-1$ the mapping $\textrm{J}_{i+1} : F_i \to F_{i+1}$ is bijective and is given by $\textrm{J}_{i+1} \big( z\big) = z - y_{i+1}$.
\item\label{yyy2} The fixed point sets $F_1 = F_m - y_1$ and $F_{i+1} = F_{i} - y_{i+1}$.
\end{enumerate}
\end{lemma}
\begin{proof}
\ref{yyy1}: Let $z$ and $\widetilde{z} $ be in $F_m$ satisfying that $\textrm{J}_1 z = \textrm{J}_1 \widetilde{z}$. Our goal is to show that $\textrm{J}_1$ is injective on $F_m$. Then, we have 
$$  z = \textrm{J}_m \textrm{J}_{m-1} \cdots \textrm{J}_{i} \textrm{J}_{i-1} \cdots \textrm{J}_1 z $$
and 
$$  \widetilde{z} = \textrm{J}_m \textrm{J}_{m-1} \cdots \textrm{J}_{i} \textrm{J}_{i-1} \cdots \textrm{J}_1 \widetilde{z}. $$
Since $\textrm{J}_1 z = \textrm{J}_1 \widetilde{z}$, then we obtain $z = \widetilde{z} $. Thus, $\textrm{J}_1$ is an injective mapping on $F_m$. Moreover, \cref{eeeesw1} shows that $\textrm{J}_1$ is a surjective mapping on $F_m$. Therefore, $\textrm{J}_1$ is a bijective mapping on $F_m$. For every $z \in F_{m}$, we have 
\begin{equation}\label{gdtw1}
z =  \textrm{J}_m \textrm{J}_{m-1} \cdots \textrm{J}_{i} \textrm{J}_{i-1} \cdots \textrm{J}_1 z.
\end{equation}
Set $z_1 = \textrm{J}_1 z, \ \ z_2 = \textrm{J}_2 z_1, \ \ \cdots, \ \ z_{m-1} = \textrm{J}_{m-1} z_{m-2}, \ \ z_m = \textrm{J}_m z_{m-1}$.
Therefore, using \cref{gdtw1}, we have $z = \textrm{J}_m z_{m-1}$ and 
$ \mathbf{z} = \big( z_1, z_2, \cdots, z_{m-1}, z\big)$. Therefore, \cref{drswas} gives 
$$\big( y_1, y_2, \cdots, y_m\big) = \big( z, z_{1}, z_{2}, \cdots, z_{m-1}\big) - \big( z_1, z_2, \cdots, z_{m-1}, z_{m}\big)$$ and therefore, $y_1 = z - z_1 \Rightarrow z_1 = z - y_1 \Rightarrow  \textrm{J}_1 z = z - y_1$. The proof of $ \textrm{J}_i$ is the same as $\textrm{J}_1$.\\
\ref{yyy2}: It follows from \ref{yyy1} that for every $z \in F_m$, we obtain $\textrm{J}_1 z = z - y_1$. Then by \cref{hygtfre1}\ref{ff2} we have $F_1 = F_{m} - y_1$. The proof for $F_{i+1} = F_{i} - y_{i+1}$ is the same as $F_1$. 
\end{proof}

\addcontentsline{toc}{section}{References}

\bibliographystyle{abbrv}

\end{document}